\documentclass[12pt]{amsart}
\usepackage{amssymb}
\usepackage{amsmath}
\usepackage{amsthm}
\usepackage{tikz-cd, amssymb}
\usepackage[hidelinks]{hyperref}
\usepackage[margin=3.1cm]{geometry}
\usepackage{fancyhdr}
\usepackage{multicol}

\usepackage[normalem]{ulem}

\DeclareMathOperator{\MSpin}{\mathrm{MSpin}}
\DeclareMathOperator{\KU}{\mathrm{KU}}
\DeclareMathOperator{\KO}{\mathrm{KO}}
\DeclareMathOperator{\K}{\mathrm{K}}
\DeclareMathOperator{\RR}{\mathbb{R}}
\DeclareMathOperator{\RP}{\mathbb{RP}}
\DeclareMathOperator{\CP}{\mathbb{CP}}

\theoremstyle{definition}
\newtheorem{theorem}{Theorem}[section]
\newtheorem{corollary}[theorem]{Corollary}
\newtheorem{remark}[theorem]{Remark}
\newtheorem{definition}[theorem]{Definition}
\newtheorem{lemma}[theorem]{Lemma}
\newtheorem{proposition}[theorem]{Proposition}

\title[Nonexistence of Green functor with values $\MSpin^c$ and $\MSpin$]{On the nonexistence of a Green functor with values spin$^c$ bordism and spin bordism}

\author{Hassan H. Abdallah}
\address{Department of Mathematics, Wayne State University, Detroit, MI, USA}
\email{hassan@wayne.edu} 
\author{Zachary Halladay}
\address{Department of Mathematics, University of Illinois at Urbana-Champaign, Urbana, IL, USA}
\email{zah2@illinois.edu} 
\author{Yigal Kamel}
\address{Department of Mathematics, University of Illinois at Urbana-Champaign, Urbana, IL, USA}
\email{ykamel2@illinois.edu}

\begin{document}

\begin{abstract}
    In this note, we show that there does not exist a $C_2$-ring spectrum whose underlying ring spectrum is $\mathrm{MSpin}^c$ and whose $C_2$-fixed point spectrum is $\mathrm{MSpin}$.
\end{abstract}

\maketitle


\section{Introduction}

In \cite{ABS}, Atiyah, Bott, and Shapiro shed light on an intimate relationship between  spin bordism and topological $\K$-theory by constructing orientations,
$$
\MSpin^c \to \KU \;\;\; \text{and} \;\;\; \MSpin \to \KO,
$$

which refine the Todd genus and the $\hat{A}$-genus, respectively, to maps of ring spectra.  Atiyah showed in \cite{AtiyahKR} that the complex and real $\K$-theory spectra can be combined into a single $C_2$-ring spectrum, $\KU_{\RR}$, called Real K-theory, with 
$$
    \KU \simeq (\KU_{\RR})^{e}  \;\;\; \text{and} \;\;\; \KO \simeq (\KU_{\RR})^{C_2}. 
$$

In \cite{HK24}, the second and third authors related the two situations above by constructing a $C_2$-ring spectrum, $\MSpin^c_{\RR}$, with ring maps,
$$
\MSpin^c \simeq (\MSpin^c_{\RR})^{e} \;\;\; \text{and} \;\;\; \MSpin \to (\MSpin^c_{\RR})^{C_2},
$$ and showed that both of the Atiyah--Bott--Shapiro orientations can be recovered from a single $C_2$-ring map,
$$
\MSpin^c_{\RR} \to \KU_{\RR},
$$
by taking underlying spectra and $C_2$-fixed points, respectively. Despite the close analogy between $\MSpin^c_{\RR}$ and $\KU_{\RR}$, it was also shown in \cite{HK24} that the natural map, $$\MSpin \to (\MSpin^c_{\RR})^{C_2},$$ is not an equivalence, and that under certain assumptions there could not exist a $C_2$-ring spectrum $E$, with $E^e \simeq \MSpin^c$ and $E^{C_2} \simeq \MSpin$. The proof in \cite{HK24} uses the integrality of the $\hat{A}$-genus on spin manifolds, and thus only rules out such an $E$ whose restriction map, $\text{res} \colon \MSpin_* \cong E^{C_2}_* \to E^e_* \cong \MSpin^c_*$, preserves the $\hat{A}$-genus of underlying oriented bordism classes of manifolds. 

\vspace{3mm}

The following theorem strengthens the aforementioned nonexistence result by making use of only the algebraic structure of the abelian groups $\MSpin_*$ and the graded ring $\MSpin^c_*$, thereby removing the extra hypothesis regarding the $\hat{A}$-genus. 


\begin{theorem}\label{intro.counterspin}
    There does not exist a genuine $C_2$-ring spectrum $E$ satisfying the following two properties.
    \begin{enumerate}
        \item The underlying ring spectrum of $E$ is equivalent to $\MSpin^c$.
        \item The $C_2$-fixed point spectrum of $E$ is equivalent to $\MSpin$.
    \end{enumerate}
\end{theorem}

\begin{corollary}
Let $\MSpin^c_{\RR}$ be the Real spin bordism spectrum constructed in \cite{HK24}. Then, $$\MSpin \not\simeq (\MSpin^c_{\RR})^{C_2} \;\;\; \text{and} \;\;\; \MSpin \not\simeq (\MSpin^c_{\RR})^{hC_2}.$$
\end{corollary}

\begin{corollary}
The Atiyah--Bott--Shapiro orientations, 
$\MSpin^c \to \KU$ and $\MSpin \to \KO$, 
do not arise as the $\{e\}$- and $C_2$-fixed point maps, respectively, of any map of $C_2$-ring spectra, $E \to \KU_{\RR}$.
\end{corollary}

\section{Proof of Theorem \ref{intro.counterspin}}

The main piece of technology that our proof makes use of is the fact that the homotopy groups of a genuine $C_2$-spectrum carry the extra structure of a Mackey functor.

\begin{definition}\label{Mackey}
    A $C_2$-\textit{Mackey Functor}, $M$, consists of Abelian groups $M^{C_2}$ and $M^e$, together with group homomorphisms,
    \vspace{-5mm}
    \begin{multicols}{2}
\[
  \begin{tikzcd}
        M^{C_2} \arrow[bend right=35,swap]{d}{\mathrm{res}} \\
        M^e \arrow[bend right=35,swap]{u}{\mathrm{tr}} \arrow[out=240,in=300,loop,swap, "\overline{(\;\;)}"]
    \end{tikzcd}
\]
  \break
  satisfying:
    \begin{enumerate}
        \item $\overline{\mathrm{res}(y)} = \mathrm{res}(y)$;
        \item $\mathrm{tr}(\overline{x}) =\mathrm{tr}(x)$;
        \item $\overline{\overline{x}} = x$;
        \item $\mathrm{res}(\mathrm{tr}(x)) = x + \overline{x}$.
    \end{enumerate}
\end{multicols}
\end{definition}

We now recall the algebraic facts about spin$^c$ and spin bordism that we will use. The following lemma collects the information we need from the ring structure of $\MSpin^c_*$.

\begin{lemma}
    There are elements $\beta \in \MSpin^c_2$ and $Z_4 \in \MSpin^c_4$, such that the subring of $\MSpin^c_*$ generated by $\beta$ and $Z_4$ is isomorphic to $\mathbb{Z}[\beta, Z_4]$.
\end{lemma}
\begin{proof}
        We recall from \cite{AS24} that there are elements $\beta = [\CP^1] \in \MSpin^c_2$ and $Z_4 \in \MSpin^c_4$ such that the subring of $\MSpin^c_*/(2)$ generated by $\beta$ and $Z_4$ is isomorphic to $\mathbb{Z}/2[\beta, Z_4]$. Furthermore, $\MSpin^c_{*}/(2)$ is isomorphic to $\mathbb{Z}/2[\beta, Z_4]$ through degree 7. Anderson--Brown--Peterson proved in \cite{ABP} that the 2-torsion (and hence all torsion) in $\MSpin^c_*$ coincides with the $\beta$-torsion, so neither $\beta$ nor $Z_4$ are 2-torsion in $\MSpin^c_{*}$. Thus, to give a presentation of the subring of $\MSpin^c_{*}$ generated by $\beta$ and $Z_4$, we need only determine if there are any relations between $\beta$ and $Z_4$. Since $Z_4$ is defined in \cite{AS24} as a lift of $[\RP^2]^2 \in \mathrm{MO}_4$ to $\MSpin^c_4$, we see that we can choose $Z_4 = [\CP^2]$. Then since $\mathbb{Z}[\mathbb{CP}^1,\mathbb{CP}^2]$ embeds into $\MSpin^{c}_{*}$ (see theorem 1.2 of \cite{BG87}) the subring of $\MSpin^c_{*}$ generated by $\beta$ and $Z_4$ is isomorphic to $\mathbb{Z}[\beta,Z_4]$.  
    \vspace{3mm}
\end{proof}

For $\MSpin_*$, the only algebraic facts we will use are:
$$
\MSpin_2 \cong \mathbb{Z}/2, \;\;\; \MSpin_4 \cong \mathbb{Z}, \;\; \text{and} \;\; \MSpin_6 = 0.
$$

\begin{proposition}\label{main.result}
There does not exist a $\mathbb{Z}$-graded $C_2$-Mackey functor,
$$
\begin{tikzcd}
        \MSpin_*  \arrow[r, bend left = 15, "\text{res}"] & \MSpin^c_* \arrow[l, bend left = 15, "\text{tr}"] \arrow[loop right, "\overline{(\:  \:)}"]
\end{tikzcd},
$$
such that $\overline{(\:\:)}: \MSpin^c_* \to \MSpin^c_*$ is a ring homomorphism.
\end{proposition}
\begin{proof}
    Suppose that a Mackey functor as in Proposition \ref{main.result} exists. By property (4) of Definition \ref{Mackey}, given any class $x \in \MSpin^c_*$, the equation,
    \[
    x + \overline{x} = \text{res}(\text{tr}(x)),
    \]
    must hold. Thus, it is sufficient to show there is some $x \in \MSpin^c_*$ for which $x + \overline{x}$ cannot be in the image of $\text{res}$. We will do this by showing that such a class must exist in at least one of the degrees $2$, $4$, or $6$. 

    \vspace{3mm}

    Since $\MSpin_2 \cong \mathbb{Z}/2$, the image of res in degree 2 must be $0$, and thus we must have that $\beta + \overline{\beta} = 0$. This gives
    \[
    \overline{\beta} = -\beta.
    \]
    Now we consider degree $4$. Using multiplicativity of the $C_2$-action, we know that $\overline{\beta^2} = \beta^2$, and so $\beta^2 + \overline{\beta^2} = 2\beta^2$. Since $\MSpin_4 \cong \mathbb{Z}$, we know that the image of res in degree 4 is either $\mathbb{Z}{\beta^2}$ or $2\mathbb{Z}{\beta^2}$. From this we know that either $Z_{4} + \overline{Z_{4}}$ is not in the image of res, or 
    \[
    Z_{4} + \overline{Z_{4}} = n\beta^2,
    \]
    for some integer $n$. We will proceed by assuming the latter and move on to considering degree $6$. Thus, we may assume
    \[
    \overline{Z_{4}} = n\beta^2 - Z_{4}.
    \]
    Then in degree $6$, we have
    \[
    \overline{\beta Z_{4}} = (\overline{\beta}) (\overline{Z_{4}}) = (-\beta)(n\beta^2 - Z_{4}) = -n\beta^3 +\beta Z_{4}.
    \]
    However,
    \[
    \beta Z_{4} +\overline{\beta Z_{4}} = -n\beta^3 + 2 \beta Z_{4} \neq 0,
    \]
    and thus $\beta Z_{4} +\overline{\beta Z_{4}}$ cannot be in the image of res, since $\MSpin_6 = 0$.
    Thus, one of the elements $\beta + \overline{\beta}$, $Z_{4} + \overline{Z_{4}}$, or $\beta Z_{4} +\overline{\beta Z_{4}}$ must not be in the image of res, and therefore such a Mackey functor cannot exist.
\end{proof}

\begin{proof}[Proof of Theorem \ref{intro.counterspin}]
    Suppose $E$ is a genuine $C_2$-ring spectrum satisfying properties (1)-(2) of Theorem \ref{intro.counterspin}. Then the homotopy Mackey functor of $E$ must take the form, 
    $$
\begin{tikzcd}
        \MSpin_*  \arrow[r, bend left = 15, "\text{res}"] & \MSpin^c_* \arrow[l, bend left = 15, "\text{tr}"] \arrow[loop right, "\overline{(\:  \:)}"]
\end{tikzcd}.
$$
Since $E$ is a $C_2$-ring spectrum, the map $\overline{(\:\:)}$ is a ring homomorphism. By Proposition \ref{main.result}, such a Mackey functor does not exist. 
\end{proof}

\section{Some remarks}

\begin{remark}
The condition in Proposition \ref{main.result} that $\overline{(\:\:)} : \MSpin^c_* \to \MSpin^c_*$ must be a ring homomorphism can evidently be slightly weakened to only requiring the ring structure to be respected through degree 6. However, some condition of this type is necessary, since for any two abelian groups $A$ and $B$, there exists a  $C_2$-Mackey functor $M$ with $M^e = A$ and $M^{C_2}=B$ by setting $\mathrm{res}=0$, $\mathrm{tr}=0$, and  $\overline{x} = -x$. 
\end{remark}

\begin{remark}
A \textit{graded $C_2$-Green functor} is a monoid object in the category of graded $C_2$-Mackey functors under the graded ``box product''. In particular, a graded Green functor, $R_*$, is a graded Mackey functor whose values are graded rings, such that the structure maps, $\mathrm{res}, 
\mathrm{tr}$, and $\overline{(\:\:)}$, satisfy certain extra properties involving the ring structures of $R_*^{C_2}$ and $R_*^{e}$. Among these is the condition that the $C_2$-action $\overline{(\:\:)} : R_*^e \to R_*^e$ is a ring homomorphism. Thus, Proposition \ref{main.result} directly implies that there does not exist a graded Green functor with values $\MSpin_*^c$ and $\MSpin_*$, which is the fact responsible for the title of this paper. Note that Theorem \ref{intro.counterspin} follows from this  statement as well.
\end{remark}

\bibliographystyle{plain} 
\bibliography{main}

\begin{thebibliography}{1}

\bibitem{AS24}
Hassan Abdallah and Andrew Salch.
\newblock Products in spin$^c$-cobordism, 2024.
\newblock \url{https://arxiv.org/abs/2407.10045}.

\bibitem{ABP}
D.~W. Anderson, E.~H. Brown, Jr., and F.~P. Peterson.
\newblock The structure of the {S}pin cobordism ring.
\newblock {\em Annals of Mathematics (2)}, 86:271--298, 1967.

\bibitem{AtiyahKR}
M.~F. Atiyah.
\newblock {K-theory and reality}.
\newblock {\em The Quarterly Journal of Mathematics}, 17(1):367--386, 01 1966.

\bibitem{ABS}
M.~F. Atiyah, R.~Bott, and A.~Shapiro.
\newblock Clifford modules.
\newblock {\em Topology}, 3:3--38, 1964.

\bibitem{BG87}
Anthony Bahri and Peter Gilkey.
\newblock The eta invariant, {${\rm Pin}^c$} bordism, and equivariant {${\rm Spin}^c$} bordism for cyclic {$2$}-groups.
\newblock {\em Pacific Journal of Mathematics}, 128(1):1--24, 1987.

\bibitem{HK24}
Zachary Halladay and Yigal Kamel.
\newblock {Real spin bordism and orientations of topological K-theory}.
\newblock {\em To appear in Transactions of the American Mathematical Society}, 2024.
\newblock \url{https://arxiv.org/abs/2405.00963}.

\end{thebibliography}

\end{document}